\documentclass[12pt]{amsart}
\usepackage{amssymb,latexsym}
\usepackage{amsfonts}
\usepackage[margin=3cm]{geometry}
\usepackage{amsmath}
\usepackage[colorlinks,linkcolor=blue,anchorcolor=blue,citecolor=blue]{hyperref}
\usepackage{algorithm}
\usepackage{enumerate}
\usepackage{algpseudocode}
\usepackage{verbatim}
\usepackage{graphicx}

\newcommand\Q{{\mathbb Q}}
\newcommand\Gal{{\mathrm {Gal}}}
\newcommand\cP{{\mathcal P}}
\newcommand\cQ{{\mathcal Q}}
\newcommand\cR{{\mathcal R}}
\newcommand\ord{\mathrm{ord}}

\newtheorem{theorem}{Theorem}[section]
\newtheorem{lemma}[theorem]{Lemma}

\newtheorem{proposition}[theorem]{Proposition}
\newtheorem{question}[theorem]{Questions}
\newtheorem{remark}[theorem]{Remark}

\numberwithin{equation}{section}

\begin{document}

\title{Primes in arithmetic progressions and nonprimitive roots}
\author{Pieter Moree}
\address{Max-Planck-Institut f\"ur Mathematik, Vivatsgasse 7, D-53111 Bonn, Germany}
\email{moree@mpim-bonn.mpg.de}

\author{Min Sha}
\address{Department of Computing, Macquarie University, Sydney, NSW 2109, Australia}
\email{shamin2010@gmail.com}

\dedicatory{Dedicated to the memory of Prof.\,Christopher Hooley (1928--2018)}



\begin{abstract}
Let $p$ be a prime. If an integer $g$ generates a subgroup of index
$t$ in $(\mathbb Z/p\mathbb Z)^*,$ then we
say that $g$ is a $t$-near primitive root modulo $p$.
We point out the easy result that each coprime residue 
class contains a 
positive natural density subset of primes $p$ not having $g$ as a $t$-near primitive root and 
prove a more difficult variant.
\end{abstract}

\maketitle

\section{Introduction}

\subsection{Background}

Given a set of primes $S$, the limit 
$$
\delta(S) = \lim_{x \to \infty} \frac{\# \{p: \, p \in S, \, p\le x\}}{\# \{p: \, p\le x\}},
$$ 
if it exists, is called the \emph{natural density}  of $S$. (Here and
in the sequel the letter $p$ is used
to denote a prime number.)

For any  integer $g \not\in \{-1,0,1\}$, let $\cP_g$ be the set of primes $p$ such that $g$ is a primitive root modulo $p$, that is
$p\nmid g$ and the \emph{multiplicative order} of $g$ modulo $p,$
$\ord_p(g),$ equals  $p-1=\# (\mathbb Z/p\mathbb Z)^*,$ and
so $g$ is a generator of $(\mathbb Z/p\mathbb Z)^*.$
In 1927, Emil Artin conjectured that 
 the set $\cP_g$ is infinite
if $g$ is not a square; moreover he also gave a 
conjectural formula for its natural density $\delta(\cP_g)$; see \cite{Mor3} for more details.
There is no explicit value of $g$ known for which $\cP_g$ can be
unconditionally proved to be infinite. However 
Heath-Brown \cite{HB}, building on
earlier fundamental work by 
Gupta and Murty \cite{GM}, showed that, given any three distinct
primes $p_1,p_2$ and $p_3,$ there is at least one
$i$ such that $\cP_{p_i}$ is infinite.

In 1967, Hooley \cite{Hooley} established Artin's conjecture under the 
 Generalized Riemann Hypothesis (GRH) and determined $\delta(\cP_g)$. Ten years later, Lenstra \cite{Lenstra} considered 
 a wide class of generalizations of Artin's conjecture.
 For example, under GRH he showed that  the primes
 in $\cP_g$ that are in a prescribed arithmetic progression 
 have a  natural density and gave a  Galois theoretic
 formula for it. This was worked out 
 explicitly by the first author \cite{Mor1, Mor2},
 who showed that  $\delta(\cP_g)=r_gA,$ with $r_g$ an 
 explicit rational number and  the Artin constant
 $$
 A = \prod_{p}\left(1 - \frac{1}{p(p-1)}\right) = 0.373955\ldots . 
 $$
 Using a powerful and 
 very general algebraic method, this
 result was rederived in a very different way by 
 Lenstra et al.\,\cite{LMS}.
 
 For any integer $t \ge 1$, let
 $$
 \cP_g(t) = \{p:~  p \nmid g, \, p \equiv 1 ~({\rm mod~}t), \, \ord_p(g) = (p-1)/t  \}.
 $$
If $p$ is in $\cP_g(t)$, then it 
is said to have $g$ as a {\it $t$-near primitive root}. 
Assuming GRH, the 
first author \cite{Mor4} determined $\delta(\cP_g(t))$
in case $g>1$ is square-free.

A more refined problem is how the primes in 
$\cP_g(t)$ are distributed over arithmetic progressions. 
To this end, let $a,d\ge 1$ be coprime integers 
and define 
 $$
 \cP_g(t,d,a) = \{p: \, p \equiv a ~({\rm mod~}d), \, p\in \cP_g(t)\}. 
 $$
 By the prime number theorem for arithmetic progressions,  
\begin{equation}
\label{pnat}
\#\{p: \,p\le x, \,  p\equiv a ~({\rm mod~}d)\}\sim \frac{x}{\varphi(d)\log x},
\end{equation}
where $\varphi$ denotes Euler's totient function. 
 A straightforward combination of the ideas used in the study of near-primitive
 roots and those for primitive roots in arithmetic progression, allows 
 one to show, assuming GRH, that $\delta(\cP_g(t,d,a))$ exists and derive a Galois theoretic expression $\delta_G(\cP_g(t,d,a))$
 for it (see Hu et al.\,\cite[Theorem 3.1]{HKMS}). 
Moreover, it can
be unconditionally shown (see \cite[Equation (3.7)]{HKMS}) that 
\begin{equation}
\label{Galoisexpression}
\limsup_{x\rightarrow \infty}\frac{\#\{p\le x: \,  p\in \cP_g(t,d,a)\}}{\pi(x)}\le \delta_G(\cP_g(t,d,a)),
\end{equation}
where as usual $\pi(x)$ denotes the prime counting function. 
The idea of the proof is to apply the simple 
asymptotic sieve up to a range in which the unconditional
Chebotarev density theorem is valid.

On the basis of
 insights from \cite{LMS}, we know that $\delta_G(\cP_g(t,d,a))$
 equals a rational multiple of the Artin constant $A,$ 
 where the rational multiple can be worked out in full
 generality. However, this is likely to produce a result
 involving several case  distinctions (as in the
 restricted case where $t=1$ and in the case where $t$ is arbitrary
 and $g$ is square-free).
 In the much less general case $g=4$ and $t=2$, the 
 expression was explicitly worked out 
 in \cite{HKMS}; see Section 
 \ref{sec:application} for more background.

\subsection{Our considerations} 
 
 In this paper we study, motivated by the
following questions, the 
distribution of primes not having a prescribed near-primitive root in arithmetic
progressions. 

\begin{question}
\label{problem}
{\rm
Let $t\ge 1$ and $g\not\in \{-1,0,1\}$ be integers.  
Let $a, d$ be 
positive coprime integers. \\
{\rm (A)} Is the set 
\begin{align*}
\cQ_g(t,d,a) = \{p: \,p \equiv a ~({\rm mod~}d), \, p\not\in \cP_g(t)\}
\end{align*}
infinite? \\
{\rm (B)} Does the set $\cQ_g(t,d,a)$ have a natural density
and can it be computed?
}
\end{question}

Since $\cP_g(t,d,a) \cup \cQ_g(t,d,a) = \{p:\, p \equiv a ~({\rm mod~}d)\}$, 
if $\delta(\cP_g(t,d,a))$ exists, then using \eqref{pnat} we have 
$$
\delta(\cQ_g(t,d,a))=1/\varphi(d)-\delta(\cP_g(t,d,a)).
$$
Question B can currently be answered only assuming GRH. 
However, in this approach it is far from evident under
which conditions on the parameters $g,t,d$ and $a$ we have
$\delta(\cQ_g(t,d,a))>0$, thus guaranteeing the infinitude
of the set $\cQ_g(t,d,a).$

Unconditionally using \eqref{Galoisexpression} we infer that
$$
\liminf_{x\rightarrow \infty}\frac{\#\{p\le x: \,  p\in \cQ_g(t,d,a)\}}{\pi(x)}\ge \frac{1}{\varphi(d)}-\delta_G(\cP_g(t,d,a)).
$$

If there exists a prime $p_0\nmid t$ satisfying
both  
$p_0 \equiv a ~({\rm mod~}d)$ and
$p_0 \not\equiv 1 ~({\rm mod~}t)$, then all the primes
 $p \equiv p_0 ~({\rm mod~}dt)$ are in $\cQ_g(t,d,a)$ (due to $t \nmid (p-1)$). 
 By \eqref{pnat}, there are infinitely many 
primes  $p \equiv p_0 ~({\rm mod~}dt)$, and they have a positive natural density.
Thus, Question A is only nontrivial when 
$p \equiv a ~({\rm mod~}d)$ implies $p\mid t$ or
$p \equiv 1 ~({\rm mod~}t),$ which is true if and only if
\begin{equation}
\label{eq:cond}
t\mid d \quad \textrm{and} \quad  t\mid (a -1). 
\end{equation}

In this note we will see 
that answering Question A is actually also rather easy in 
case \eqref{eq:cond} is satisfied. The answer to Question A is yes, and we can be even a little bit
more precise on using Kummerian extensions of cyclotomic 
number fields $\mathbb Q(\zeta_n)$ with
 $\zeta_n=e^{2\pi i/n}$.
 
\begin{proposition}
\label{elementary} 
Let $g \not\in \{-1,0,1\}$ and $t\ge 1$ 
be integers. Let $a, d$ be 
positive coprime integers. 
Then, for any integer $q > 2$ coprime to $2dt$, the set $\cQ_g(t,d,a)$ contains a positive natural density 
subset of primes $p$ having natural density 
$$\frac{1}{[\Q(\zeta_d,\zeta_{q},g^{1/q}):\Q]}.$$
\end{proposition}

The field degree 
$[\Q(\zeta_d,\zeta_{q},g^{1/q}):\Q]=[\Q(\zeta_{\text{lcm}(d,q)},g^{1/q}):\Q]$ is not difficult to compute for any given $g, d$ and $q$;  see \cite[Lemma 1]{Moree2005} for the general result 
(which is a direct consequence of \cite[Proposition 4.1]{Wagstaff}). 
Using this computation the maximum density
of the $q$-dependent subsets arising in Proposition \ref{elementary} can be determined; 
see the next section for an example. If $\ell$  is a prime factor of $q$,
then $\Q(\zeta_d,\zeta_{\ell},g^{1/\ell})
\subseteq \Q(\zeta_d,\zeta_{q},g^{1/q})$, and so a priori
the maximum occurs in an odd prime.

We will also
establish a more difficult
variant of Proposition \ref{elementary}. 
Letting $g,t,d,a$ be as in Proposition \ref{elementary}, we define 
the set 
$$
\cR_g(t,d,a)= \{p:~  p \nmid g, \, p \equiv a ~({\rm mod~}d), \, p \equiv 1 ~({\rm mod~}t),    
\,\ord_p(g) \mid  (p-1)/t \}. 
$$ 
Clearly, we have $\cP_g(t,d,a) \subseteq \cR_g(t,d,a)$.   
Our purpose is to show that if $\cR_g(t,d,a)$ is not empty, then 
$\cR_g(t,d,a)$ contains a positive density subset of primes not contained in $\cP_g(t,d,a)$.  

\begin{theorem}
\label{main}
Let $g \not\in \{-1,0,1\}$ and $t\ge 1$ 
be integers. Let $a, d$ be 
positive coprime integers. 
Suppose the set $\cR_g(t,d,a)$ is not empty.  
Then, for any integer $q > 2$ coprime to $2dgt$, the set $\cR_g(t,d,a)$ contains a 
subset of primes $p$ 
for which $g$ is a non $t$-near primitive root modulo $p$ having natural density 
$$\frac{1}{[\Q(\zeta_d,\zeta_{qt},g^{1/qt}):\Q]}.$$
\end{theorem}

Again, given $d, g$ and $t,$ the maximum density of the 
$q$-dependent subsets arising
in the theorem can be determined, and for this it suffices to consider
primes $q\nmid 2dgt$.

Note that for any integer $q \ge 2$, each prime in $\cR_g(qt,d,a)$ is not contained in $\cP_g(t,d,a)$. 
So, Theorem~\ref{main} is derived directly from the following proposition, which might be of independent interest. 

\begin{proposition}
\label{Rg} 
Let $g \not\in \{-1,0,1\}$ and $t\ge 1$ 
be integers. Let $a, d$ be positive coprime integers. 
Suppose the set $\cR_g(t,d,a)$ is not empty.  
Then, for any positive integer $q$ coprime to $2dgt$, we have 
$$
\delta(\cR_g(qt,d,a))=
\frac{1}{[\Q(\zeta_d,\zeta_{qt},g^{1/qt}) : \Q]}. 
$$
\end{proposition}

\subsection{An application} 
\label{sec:application}
Proposition \ref{elementary} has an application to 
\emph{Genocchi numbers} $G_n,$ which are defined by 
$G_n = 2(1-2^n)B_n,$
where $B_n$ is the $n^{\text{th}}$ Bernoulli number. The Genocchi numbers are actually integers.
As introduced in \cite{HKMS}, if a prime $p>3$ divides at least one of the Genocchi numbers $G_2,G_4,\ldots, G_{p-3},$ it
is said to be \emph{$G$-irregular}
and \emph{$G$-regular} otherwise. 
The first fifteen G-irregular primes \cite{OEIS} are  
$$
17, 31, 37, 41, 43, 59, 67, 73, 89, 97, 101, 103, 109, 113, 127.
$$
The G-regularity of primes can 
be linked to the divisibility of certain class numbers of cyclotomic fields.  
Let $S$ be the set of infinite places of $\Q(\zeta_p)$ and $T$ the set of places above the prime 2. 
Denote by $h_{p,2}$ the \textit{$(S,T)$-refined class number} of $\Q(\zeta_p)$ and $h_{p,2}^{+}$ be the refined class number of $\Q(\zeta_p+\zeta_p^{-1})$ with respect to its infinite places and places above the prime 2
(for the definition of the refined class number of global fields, see for example Hu and Kim \cite[Section 2]{HK}). 
Define $h_{p,2}^{-} = h_{p,2} / h_{p,2}^{+}.$
It turns out that $h_{p,2}^{-}$ is an integer
(see \cite[Proof of Proposition 3.4]{HK}).
Recall that a \textit{Wieferich prime} is an odd prime $p$ such that 
$2^{p-1}\equiv 1~({\rm mod~}p^2).$
\begin{theorem}{\rm \cite[Theorem 1.5]{HKMS}}.
\label{thm:class}
 Let $p$ be an odd prime. 
 Then, if $p$ is G-irregular, we have $p\mid h_{p,2}^{-}$. 
 If furthermore  $p$ is not a Wieferich prime, the converse is also true. 
\end{theorem}

It is easy to show that if $\ord_p(4) \ne (p-1)/2,$ then $p$ is G-irregular; see \cite[Theorem 1.6]{HKMS}.
Hence, taking $g=4$ and $t=2$ in Proposition~\ref{elementary} 
and noting that we have 
$[\Q(\zeta_{d},\zeta_{q},4^{1/q}):\Q]=\varphi(d)q(q-1)$ for any prime $q \nmid 2d$, we 
arrive at the following result.
\begin{proposition}
Let $a, d$ be positive coprime integers.
Let $q$ be the smallest prime not dividing  $2d$. 
The set of G-irregular primes $p$ satisfying $p\equiv a ~({\rm mod~}d)$ 
contains a subset having natural density 
$$
\frac{1}{\varphi(d)q(q-1)}.
$$
\end{proposition}
This result is a weaker version of Theorem 1.11 in \cite{HKMS}, 
however its proof is much more elementary, and it still shows that each coprime residue class contains a subset of G-irregular primes 
having positive natural density.

\section{Preliminaries}
\label{sec:pre}
Given any integers $d,n\ge 1$ put
$K_n =\Q(\zeta_d,\zeta_{n},g^{1/n}).$
For $a$ coprime to $d,$ let $\sigma_a$ be the endomorphism of $\Q(\zeta_d)$ over $\Q$ defined by $\sigma_a (\zeta_d) = \zeta_d^{a}$. 
Let $C_n$ be the conjugacy class of elements of the Galois group $G_n =\Gal(K_n/\Q)$ such that for any $\tau_n \in C_n$, 
\begin{equation}
\label{compatible}
\tau_n \big |_{\Q(\zeta_d)} = \sigma_a, \qquad  \tau_n \big |_{\Q(\zeta_{n}, g^{1/n})} = \text{id}, 
\end{equation}
where `id' stands for the identity map.
Note that either $C_n$ is empty, or $C_n$ is nonempty and $|C_n|=1$. The latter case occurs if and only if 
\begin{equation}
\label{compatible2}
\tau_n \big |_{\Q(\zeta_d)\cap \Q(\zeta_{n}, g^{1/n})} = \text{id}.
\end{equation}
If this condition is satisfied,
then by the Chebotarev density theorem (in its
natural density form, cf. Serre \cite{Serre}, 
the original form being for
Dirichlet density),
the  primes unramified in $K_n$ and with Frobenius $C_n$ have natural density $1/[K_n:\Q]$.  
Note that the  primes unramified in $K_n$ are exactly the primes  $p \nmid dgn$. 
The first condition on $\tau_n$ ensures that the primes
$p \nmid dgn$ having $\tau_n$ as Frobenius satisfy 
$p \equiv a ~({\rm mod~}d)$. Likewise the second condition
ensures that such primes satisfy $\ord_p(g)\mid (p-1)/n$.

In particular, in case  $\Q(\zeta_d)$ and 
$\Q(\zeta_{n}, g^{1/n})$ are linearly disjoint over $\Q$, that is, 
\begin{equation}
\label{disjoint}
\Q(\zeta_d)\cap  
\Q(\zeta_{n}, g^{1/n})=\Q,
\end{equation}
we have $|C_n|=1$, and
 the primes $p \nmid dgn$ with
Frobenius $C_n$ satisfy 
$p \equiv a ~({\rm mod~}d)$ and 
$\ord_p(g)\mid (p-1)/n$, and they have natural density $1/[K_n:\Q]$.

\section{Proofs} 

\subsection{Proof of Proposition~\ref{elementary}}
Since
$q$ is odd, the extension $\Q(\zeta_{q}, g^{1/q})$ of $\Q(\zeta_q)$ is nonabelian
and $$\Q(\zeta_d)\cap  
\Q(\zeta_{q}, g^{1/q})=\Q(\zeta_d)\cap \Q(\zeta_q)=\Q(\zeta_{\gcd(d,q)})=\Q,$$
as $\gcd(q,d)=1$.  
Thus \eqref{disjoint} is satisfied 
and consequently there is 
a set with natural density $1/[K_q:\Q]$ of
primes $p$ satisfying 
$p \equiv a ~({\rm mod~}d)$ and 
$\ord_p(g)\mid (p-1)/q.$ Since by
assumption $q \nmid t$, it follows that for these
primes $p$, $\ord_p(g)\ne (p-1)/t$, and so for
them $g$ is a non $t$-near primitive root. 
This completes the proof.

\subsection{Proof of Proposition~\ref{Rg}}
From now on we assume that $g,t, a$ and $d$ are 
as in Proposition~\ref{Rg}. The proof of Proposition~\ref{Rg} 
rests on the Chebotarev density theorem
and the following lemma. Recall that $K_n =\Q(\zeta_d,\zeta_{n},g^{1/n})$. 

\begin{lemma}
\label{doorsnede}
Put 
$
I_n = \Q(\zeta_d) \cap \Q(\zeta_{n}, g^{1/n}).
$ 
Then, for any positive integer $q$ coprime to $2dgt$, 
we have $I_{qt}=I_t$.  
\end{lemma}

\begin{proof}
Since $I_t \subseteq I_{qt}$, it suffices to show that $[I_{qt}: \Q] = [I_t : \Q]$.  
Obviously $[d,t]=rt$ for some positive integer $r$. 
By elementary Galois theory and noticing that $\gcd(q,dt)=1$, we 
see that
$$
[I_{qt} : \Q] = \frac{[\Q(\zeta_d) : \Q] \cdot [\Q(\zeta_{qt},g^{1/qt}) : \Q]}{[\Q(\zeta_d, \zeta_{qt},g^{1/qt}) : \Q]} 
 = \frac{\varphi(d) [\Q(\zeta_{qt},g^{1/qt}) : \Q]}{[\Q(\zeta_{qrt},g^{1/qt}) : \Q]},
$$
and, similarly,
$
[I_t : \Q] = \varphi(d) [\Q(\zeta_{t},g^{1/t}) : \Q]/[\Q(\zeta_{rt},g^{1/t}) : \Q].
$
Then, by Lemma 1 of \cite{Moree2005} and noticing $\gcd(q,2dgt)=1$,   
it is straightforward to deduce that 
$[I_{qt}: \Q] = [I_t : \Q]$. 
\end{proof}

\begin{remark}
{\rm
We remark that the condition $\gcd(q, 2dgt)=1$ cannot be removed. 
For example, choosing $g=21, d=3, t=10, q=7$ and using \cite[Lemma 2.4]{Mor2}, 
we have $I_t = \Q$ and $I_{qt} = \Q(\zeta_d)=\Q(\sqrt{-3})\ne I_t$. 
}
\end{remark}

\begin{proof}[Proof of Proposition~\ref{Rg}] 
By Lemma \ref{doorsnede} it follows that 
\begin{equation}
\label{iqi1}
I_{qt}=I_t.
\end{equation}
By assumption,  $\cR_g(t,d,a)$ is not empty.
Then, this implies that the two automorphisms
in \eqref{compatible} are compatible
and hence \eqref{compatible2} is satisfied,
which leads to the conclusion that
$\cR_g(t,d,a)$ is not only nonempty, but 
even has a positive natural density, moreover
 $\delta(\cR_g(t,d,a)) = [K_{t}:\Q]^{-1}$ by the discussions in Section~\ref{sec:pre}. 
So, there must be a $\tau_t\in C_t$ such that
$\tau_t |_{I_t}=\text{id},$ which by \eqref{iqi1} implies
the existence of an automorphism $\tau_{qt} \in C_{qt}$ such that
$\tau_{qt} |_{I_{qt}}=\text{id}$. Then, it follows from 
the discussions in Section~\ref{sec:pre} that
$\delta(\cR_g(qt,d,a))=[K_{qt}:\mathbb Q]^{-1}$. 
\end{proof}

\subsection{Proof of Theorem~\ref{main}}
\begin{proof}[Proof of Theorem~\ref{main}] 
A direct consequence of Proposition ~\ref{Rg}.
\end{proof}

\section*{Acknowledgement}

The research of Min Sha was supported by a Macquarie University Research Fellowship. 
The authors thank Christine McMeekin and Peter Stevenhagen for very helpful discussions and feedback. The latter pointed
out a serious oversight in an earlier version. The 
feedback of the referee
is also gratefully acknowledged.


\begin{thebibliography}{99}


\bibitem{OEIS} On-line encyclopedia of integer sequences, \textit{Genocchi irregular primes}, 
sequence A321217, available at \url{https://oeis.org/A321217}. 

\bibitem{GM}
R.\,Gupta and M.R.\,Murty, \textit{A remark on Artin's conjecture}, Invent. Math. \textbf{78} (1984), 127--130.


\bibitem{HB} D.R.\,Heath-Brown, 
\textit{Artin's conjecture for primitive roots}, Quart. J. Math. Oxford Ser. (2) {\bf 37} (1986), 
27--38. 

\bibitem{Hooley}
C.\,Hooley, \textit{On Artin's conjecture}, J. Reine Angew. Math. \textbf{225} (1967), 209--220. 

\bibitem{HK}
S.\,Hu and M.-S.\,Kim, \textit{The $(S,\{2\})$-Iwasawa theory}, 
J. Number Theory {\bf 158} (2016), 73--89.

\bibitem{HKMS}
S.\,Hu, M.-S.\,Kim, P.\,Moree and M.\,Sha, \textit{Irregular primes with respect to Genocchi numbers and 
Artin's primitive root conjecture}, J. Number Theory, to appear, available at \url{https://arxiv.org/abs/1809.08431}. 

\bibitem{Lenstra} 
H.W.\,Lenstra, Jr., \textit{On Artin's conjecture and Euclid's algorithm in global fields}, Invent. Math. {\bf 42} (1977), 202--224.

\bibitem{LMS} 
H.W.\,Lenstra, Jr., P.\,Moree and P.\,Stevenhagen, \textit{Character sums for primitive root densities}, 
Math. Proc. Cambridge Philos. Soc. {\bf 157} (2014), 489--511.

\bibitem{Mor1} 
P.\,Moree, \textit{On primes in arithmetic progression
having a prescribed primitive root}, J. Number Theory
{\bf 78} (1999), 85--98.

\bibitem{Moree2005}
P.\,Moree, \textit{On the distribution of the order and index of $g~({\rm mod}~ p)$ over residue classes I}, 
J. Number Theory \textbf{114} (2005), 238--271.

\bibitem{Mor2} 
P. Moree, \textit{On primes in arithmetic progression having a 
prescribed primitive root II},  Funct. Approx. Comment. Math. \textbf{39} (2008), 133--144.

\bibitem{Mor3}
P.\,Moree, \textit{Artin's primitive root conjecture $-$ a survey}, Integers \textbf{12A} (2012),  A13.

\bibitem{Mor4} 
P.\,Moree, \textit{Near-primitive roots},  Funct. Approx. Comment. Math. \textbf{48} (2013), 133--145. 

\bibitem{Serre} J.-P.\,Serre, 
\textit{Quelques applications du th\'eor\`eme de densit\'e de 
Chebotarev}, Inst. Hautes \'Etudes Sci. Publ. Math. {\bf 54} (1981), 
323--401. 


\bibitem{Wagstaff} 
S.S.\,Wagstaff, Jr., \textit{Pseudoprimes and a generalization of Artin's conjecture},
 Acta Arith. {\bf 41} (1982), 141--150.


\end{thebibliography}
\end{document}